\documentclass[12pt]{article}
\usepackage{amscd}
\usepackage[all]{xy}
\usepackage{bbm}
\usepackage{amssymb}
\usepackage{latexsym}
\usepackage{amsmath}
\usepackage{amsthm}

\newtheorem{Theorem}{Theorem}[section]
\newtheorem{Lemma}[Theorem]{Lemma}

\newtheorem{Proposition}[Theorem]{Proposition}
\newtheorem{Remark}[Theorem]{Remark}

\begin{document}

\begin{center}
{\large\bf Remarks on minimal rational curves on moduli spaces of stable bundles \ }

\vspace{0.2cm}

Liu Min \\
College of Mathematics, Qingdao University,  Qingdao 266071, P.R.China\\
E-mail: liumin@amss.ac.cn\\
\end{center}

\footnotetext{ Supported by the National Natural Science Foundation of China (Grant No. 11401330) }

\begin{abstract}
Let $M$ be the moduli space of stable bundles of rank 2 and with
fixed determinant $\mathcal{L}$ of degree $d$ on a smooth projective
curve $C$ of genus $g\geq 2$. When $g=3$ and $d$ is even, we prove,
for any point $[W]\in M$, there is a minimal rational curve passing
through $[W]$, which is not a Hecke curve. This complements a
theorem of Xiaotao Sun.
\end{abstract}

\section{Introduction}

Let $C$ be a smooth projective curve of genus $g\geq 2$ and $\mathcal{L}$ a line bundle on $C$ of degree $d$. Let
$M:=SU_{C}(r,\mathcal{L})$ be the moduli space of stable vector bundles of rank $r$ and with the fixed determinant $\mathcal{L}$,
which is a smooth quasi-projective Fano-variety with $Pic(M)=\mathbb{Z}\cdot\Theta$ and $-K_{M}=2(r,d)\Theta$,
where $\Theta$ is an ample divisor(\cite{Ramanan} \cite{DrezetNarasimhan}).
For any rational curve $\phi: \mathbb{P}^{1}\rightarrow M$, we can define its degree $\text{deg}\phi^{*}(-K_{M})$ with
respect to the ample anti-canonical line bundle $-K_{M}$.

In \cite{Sun05}, Xiaotao Sun determines all the rational curves of
minimal degree passing through a generic point of $M$ except in the
case of $g=3, \ r=2$ and $d$ is even.

\begin{Theorem}(Theorem 1 of \cite{Sun05})\label{thm1.1} If $g\geq
3$, then any rational curve $\phi: \mathbb{P}^{1}\rightarrow M$
passing through the generic point has degree at least $2r$. It has
degree $2r$ if and only if it is a Hecke curve \textbf{unless $g=3$,
$r=2$, and $d$ is even.}
\end{Theorem}

It implies that all the rational curves of $(-K_{M})$-degree smaller
than $2r$, called \emph{small rational curves}, must lie in a proper
closed subset( \cite{LM2012}, \cite{MokSun09}). In this note, we
remark that the condition in Sun's Theorem is necessary:

\begin{Theorem}\label{thm1.2} If $g=3$, $r=2$ and $d$ is even, then, for any point
$[W]\in M$, there exists a rational curve of degree $4$ passing
through it, which is not a Hecke curve.
\end{Theorem}

Recall that, by Lemma 2.1 of \cite{Sun05}, any rational curve
$\phi:\mathbb{P}^{1}\to M$ is defined by a vector bundle $E$ on $f:
X=C\times\mathbb{P}^{1}\to C$. \textbf{If $E$ is semi-stable on
generic fiber} $X_{\xi}=f^{-1}(\xi)$, according to the arguments of
\cite{Sun05}, there is a finite set $S\subset C$ of points and a
vector bundle $V$ on $C$ such that $E$ is suited in the exact
sequence $$0\to f^*V\to E\to \bigoplus_{p\in S}\mathcal{Q}_p\to 0$$
where $\mathcal{Q}_p$ is a vector bundle on
$X_p=\{p\}\times\mathbb{P}^{1}$. The curves defined by such $E$ was
called of \textbf{Hecke type }in \cite{Sun2011} (since a Hecke curve
by definition is defined by a vector bundle $E$ suited in $0\to
f^*V\to E\to \mathcal{O}_{X_p}(-1)\to 0$). \textbf{If $E$ is not
semi-stable on generic fiber} $X_{\xi}$ (curves defined by such $E$
was called of \textbf{split type} in \cite{Sun2011}) and the curve
has minimal degree (i.e. a line), then $E$ must suite in
$$0\to f^*V_1\otimes\pi^*\mathcal{O}_{\mathbb{P}^1}(1)\to E\to
f^*V_2\to 0$$ where $\pi:X\to \mathbb{P}^1$ is the projection and
$V_1$, $V_2$ are stable vector bundles on $C$ of rank $r_1$, $r_2$
and degree $d_1$, $d_2$ satisfying $r_1d-rd_1=(r,d)$.

The rational curves we constructed in Theorem \ref{thm1.2} are of
split type (thus they are not Hecke curves). We in fact have a more
general result. Let $M=\mathcal{S}U_C(2,\mathcal{L})$ be the moduli
space of rank two stable bundles with fixed determinant
$\mathcal{L}$ on a smooth projective curve $C$ of genus $g\ge 3$.
Let $M_s\subset M$ be the locus of stable bundles $[W]\in M$ with
Segre invariant $s(W)=s$. Then we have

\begin{Theorem}\label{thm1.3} When $d$ is even, for any $[W]\in
M_2$, there is a rational curve of split type passing through it,
which has degree $4$. If $d$ is odd, for any $[W]\in M_1$, there is
a rational curve of split type passing through it, which has degree
$2$.
\end{Theorem}

  When $g=3$ and $d$ is even, we have $M_2=M$ (see Lemma 3.1). Thus Theorem
  \ref{thm1.2} is a corollary of Theorem 1.3.

\section{Rational curves of split type}

Let $C$ be a smooth irreducible projective curve with genus $g\geq2$ over an algebraically closed field,
 $W$ be a stable bundle of rank $r$
and with determinant $\mathcal{L}$
over $C$. If there is a stable subbundle $V_{1}$ of $W$ such that
\begin{equation}
r_{1}d-d_{1}r=(r, d),
\label{eq:(r1,d1)}
\end{equation}
where $r_{1}=\text{rank}V_{1}$, $d_{1}=\text{deg}V_{1}$ and $d=\text{deg}W$. Let $V_{2}:=V/V_{1}$ be the quotient  bundle, then $W$ fits a
non-trivial extension
\begin{equation}
 0\rightarrow V_{1}\rightarrow W\rightarrow V_{2}\rightarrow 0.
\label{eq:p0}
\end{equation}

 It is known that there is a family of vector bundles $\{\mathcal{E}_{p}\}_{p\in P}$ on $C$ parametrized by $P=\mathbb{P}Ext^{1}(V_{2},V_{1})$
 so that for each $p\in P$, the $\mathcal{E}_{p}$ is isomorphic to the bundle obtained as the extension of $V_{2}$ by $V_{1}$ given by $p$
 ( see Lemma 2.3 of \cite{Ramanan} ). Let $l$ be a line in $P=\mathbb{P}Ext^{1}(V_{2},V_{1})$ passing through the point $p_{0}$,
 where $p_{0}$ is the point in $P$ given by (\ref{eq:p0}). If it happens that $\mathcal{E}_{p}$ is stable for each $p\in l$, then
 $$\{\mathcal{E}_{p}\}_{p\in l}$$
will define a rational curve of degree $2(r,d)$ (with respect to
$-K_{M}$) passing through $[W]\in SU_{C}(r,\mathcal{L})$
(\cite{Sun05}, \cite{MokSun09}). Such a rational curve in
$SU_{C}(r,\mathcal{L})$ will be called a \textbf{rational curve of
split type}.

It is known that an extension $0\rightarrow E \rightarrow W
\rightarrow F \rightarrow 0$, where $E, W, F$ are vector bundles on
$C$, gives rise to an element $\delta(W)\in H^{1}(C, Hom(F,E))$
which is the image of the identity homomorphism in $H^{0}(C,
Hom(F,F))$ by the connecting homomorphism $H^{0}(C,
Hom(F,F))\rightarrow H^{1}(C, Hom(F,E))$. This gives a one-one
correspondence between the set of equivalent classes of extensions
of $F$ by $E$ and $H^{1}(C, Hom(F,E))$ (\cite{Ramanan}).

\begin{Lemma}
Let $d$ be an even number, and $0\rightarrow L_{1}\rightarrow W \rightarrow L_{2}\rightarrow 0$ be any non-trivial extension of $L_{2}$ by $L_{1}$,
where $L_{1}$ ( resp. $L_{2}$ ) is a line bundle of degree $\frac{d}{2}-1$ ( resp. $\frac{d}{2}+1$ ). Then

(i) $W$ is semi-stable;

(ii) $W$ is non-stable if and only if the element $\delta(W)\in H^{1}(C, L_{2}^{-1}\otimes L_{1})$ corresponding to $W$ is in the kernel of the map
$$ H^{1}(C, L_{2}^{-1}\otimes L_{1})\longrightarrow H^{1}(C, L_{2}^{-1}\otimes L_{1}\otimes L_{x}),$$
for some $x\in C$. In this case , $W$ is $S$-equivalent to $L_{2}\otimes L_{x}^{-1}\oplus L_{1}\otimes L_{x}$.

\label{lm:non-stable of extension}
\end{Lemma}

\begin{proof}
(i) See Lemma 2.2 in \cite{MokSun09} and \cite{MokSun14}.

(ii) Let $L'$ be a line bundle of degree $\frac{d}{2}$. Then, since $H^{0}(C,\text{Hom}(L', L_{1}))=0$,
 it is easy to see that $H^{0}(C,\text{Hom}(L',W))\neq 0$ if and only if $L'$ is of the form $L_{2}\otimes L_{x}^{-1}$ for some $x\in C$ and the natural map $L_{2}\otimes L_{x}^{-1}\rightarrow L_{2}$ can be lifted into a map $L_{2}\otimes L_{x}^{-1}\rightarrow W$.

 Consider the commutative diagram of vector bundles
 \[\begin{CD}
0\rightarrow Hom(L_{2},L_{1})@>>>Hom(L_{2},W)@>>>Hom(L_{2},L_{2})\rightarrow0\\
@VVV@VVV@VVV\\
0\rightarrow Hom(L_{2}\otimes L_{x}^{-1},L_{1})@>>>Hom(L_{2}\otimes L_{x}^{-1}, W)@>>>Hom(L_{2}\otimes L_{x}^{-1},L_{2})\rightarrow0,
\end{CD} \]
 where the horizontal sequences are exact and the vertical maps ar induced by the natural map $L_{2}\otimes L_{x}^{-1}\rightarrow L_{2}$ . From this we deduce the commutative diagram
  \[\scriptsize\begin{CD}
0\rightarrow H^{0}(C,Hom(L_{2},W))@>>>H^{0}(C,Hom(L_{2},L_{2}))@>>>H^{1}(C,Hom(L_{2},L_{1}))\rightarrow\cdots\\
@VVV@VVV@VVV\\
0\rightarrow H^{0}(C,Hom(L_{2}\otimes L_{x}^{-1}, W))@>>>H^{0}(C,Hom(L_{2}\otimes L_{x}^{-1},L_{2}))@>>>H^{1}(C,Hom(L_{2}\otimes L_{x}^{-1},L_{1}))\rightarrow\cdots
\end{CD} \]
 Which implies the lemma.
\end{proof}

\begin{Remark} Lemma \ref{lm:non-stable of extension} (ii) asserts that the non-stable bundles in $PH^{1}(L_{2}^{-1}\otimes L_{1})$
corresponds precisely to the image of $C$ in $PH^{1}(L_{2}^{-1}\otimes L_{1})$ under the map given by the linear system
$K\otimes L_{1}^{-1}\otimes L_{2} $. Which implies that the dimension of the subset of non-stable bundles in $PH^{1}(L_{2}^{-1}\otimes L_{1})$
is at most 1.
 \label{rk:extension of line bundles}
\end{Remark}

\section{The Proof of Theorem 1.3}

 Let $C$ be a smooth irreducible curve over an algebraically closed field of characteristic zero,
  $W$ a vector bundle of rank $2$ over $C$, set
\begin{equation}
m(W):=max\{\text{deg}(L)| L\subset W \text{ is a sub line bundle of } W\},
\label{def:m(W)}
\end{equation}
where the maximum is taken over all sub line bundles $L$ of $W$. A sub line bundle $L$ of $W$ of maximal degree $m(W)$ is called a
\textbf{maximal sub line bundle}.

The \textbf{Segre invariant} is defined by
\begin{equation}
s(W):=\text{deg}(W)-2m(W).
\label{def:s(w)}
\end{equation}
 Note that $s(W)\equiv\text{deg}(W)\ (\text{mod }2)$ and that $W$ is stable ( resp. semi-stable ) if and only if $s(W)\geq1$ ( resp. $s(W)\geq0$ ).
 Nagata proved in \cite{Nagata70} that $$s(W)\leq g.$$

\begin{Lemma}
 If $g=3$, then, for any stable bundle $W$ over $C$ of rank 2 and with even degree $d$, we have $s(W)=2$.
\label{lm:s(w)=2}
\end{Lemma}

\begin{proof}
Since $W$ is stable, we have $1\leq s(W)$ and $s(W)\leq g=3$ by
Nagata's Theorem (\cite{Nagata70}). At the same time, since $d$ is
even , it is easy to see that $s(W)\equiv 0 \text{ mod 2 }$ by the
definition. Thus we must have $s(W)=2$.
\end{proof}

In general, the function $s: M\longrightarrow \mathbb{Z}$ defined by
$[W]\longmapsto s(W)$ is lower semicontinuous and this gives a
stratification of $M$ into locally closed subsets $M_{s}$ according
to the value of $s$. Then, by Proposition 3.1 in
\cite{LangNarasimhan83}, we have

\begin{Proposition}(\cite{LangNarasimhan83})
Suppose $1\leq s\leq g-2$ and $s\equiv d (\text{ mod } 2)$. Then
$M_{s}$ is an irreducible algebraic variety of dimension 2g+s-2.
\label{prop:dim(Ms)}
\end{Proposition}

The proof of Theorem 1.3 follows the following two propositions.

\begin{Proposition}
Suppose $g\geq3$, $r=2$, $d$ is even and $M_{2}$ is non-empty. Then,
for any $[W]\in M_{2}$, there is a rational curve of split type
passing through it, which has degree $4$. \label{prop:M2}
\end{Proposition}

\begin{proof}
For any $[W]\in M_2$, there is a sub line bundle $L_{1}$ of $W$ with
$\text{deg}L_{1}=\frac{d}{2}-1$, where $d=\text{deg}\mathcal{L}$.
Let $L_{2}:=W/L_{1}$ be the quotient bundle, which has degree
$\frac{d}{2}+1$. It is easy to see that
$$1\times d-(\frac{d}{2}-1)\times 2=2=(2,d).$$
Let $i: L_{1}\rightarrow W$ be the natural injection, then
\[\begin{CD}
0@>>>L_{1}@>i>>W@>>>L_{2}@>>>0
\end{CD} \]
is a non-trivial extension (otherwise, we have $W\cong L_{1}\oplus
L_{2}$, which contradicts to the stability of $W$).

It is known that there is a family of vector bundles $\mathcal{E}$
on $C$ parametrized by $P_{(L_{1},L_{2})}=\mathbb{P}Ext^{1}(L_{2},
L_{1})$ so that for each $p\in P_{(L_{1},L_{2})}$, the
$\mathcal{E}_{p}$ is isomorphic to the bundle obtained as the
extension of $L_{2}$ by $L_{1}$ given by $p$ ( see Lemma 2.3 of
\cite{Ramanan} ). More precisely, there is a universal extension
\begin{equation}
0\rightarrow f^{*}L_{1}\otimes \pi^{*}\mathcal{O}_{P_{(L_{1},L_{2})}}(1)\rightarrow \mathcal{E}\rightarrow f^{*}L_{2}\rightarrow 0
\label{eq:2.4}
\end{equation}
on $C\times P_{(L_{1},L_{2})}$, where $f: C\times P_{(L_{1},L_{2})}\rightarrow C$ and $\pi: C\times P_{(L_{1},L_{2})}\rightarrow P_{(L_{1},L_{2})}$
are projections. Then $\mathcal{E}$ is a family of semi-stable bundles of rank $r$ and with fixed determinant
$\text{det}(L_{1})\otimes \text{det}(L_{2})\cong\mathcal{L}$ ( Lemma \ref{lm:non-stable of extension}). Thus, the universal extension (\ref{eq:2.4})
defines a morphism
\begin{equation}
\Phi_{(L_{1}, L_{2})}:\ P_{(L_{1},L_{2})}\longrightarrow U_{C}(2,\mathcal{L}),
\label{morphism:phi}
\end{equation}
where $U_{C}(2,\mathcal{L})$ denotes the moduli space of semi-stable bundles of rank 2 and with fixed determinant $\mathcal{L}$, which is a projective
closure of $M$.

It is easy to see that $P_{(L_{1},L_{2})}$ is a projective space of
dimension $g\ge 3$. By Lemma \ref{lm:non-stable of extension} and
Remark \ref{rk:extension of line bundles}, there is a line $l$ in
$P_{(L_{1},L_{2})}$ passing through
 \[\begin{CD}
q=[0@>>>L_{1}@>i>>W@>>>L_{2}@>>>0]
\end{CD} \]
such that $\mathcal{E}_{p}$ is stable for each $p\in l$. Thus,
$\Phi_{(L_{1}, L_{2})}(l)\subset M=SU_{C}(2,\mathcal{L})$ and
\begin{equation}
\Phi_{(L_{1}, L_{2})}|_{l}: l\rightarrow M=SU_{C}(2,\mathcal{L})
\end{equation}
is a rational curve of split type passing through the point $[W]\in
M$.
\end{proof}

\begin{Proposition}
Suppose $g\geq 2$, $r=2$, $d$ is odd and $M_{1}$ is non-empty. Then,
for any $[W]\in M_{1}$, there is a rational curve of split type
passing through it, which has degree $2$. \label{prop:M1}
\end{Proposition}

\begin{proof}
Let $[W]$ be a point in $M_{1}$, then we have $s(W)=1$ and there is
a sub line bundle $L_{1}$ of $W$ with
$\text{deg}L_{1}=\frac{d-1}{2}$, where $d=\text{deg}\mathcal{L}$.
Let $L_{2}:=W/L_{1}$, which is a line bundle of degree
$\frac{d+1}{2}$. It is easy to see that
$$1\times d-\frac{d-1}{2}\times 2=1=(2,d).$$
Let $\iota: L_{1}\rightarrow W$ be the natural injection, then
\[\begin{CD}
0@>>>L_{1}@>\iota>>W@>>>L_{2}@>>>0
\end{CD} \]
is a non-trivial extension because $W$ is a stable bundle.

It is known that there's a family of vector bundles $\mathcal{E}$ on
$C$ parametrized by $P_{(L_{1},L_{2})}=\mathbb{P}Ext^{1}(L_{2},
L_{1})$ so that for each $p\in P_{(L_{1},L_{2})}$, the
$\mathcal{E}_{p}$ is isomorphic to the bundle obtained as the
extension of $L_{2}$ by $L_{1}$ given by $p$ ( see Lemma 2.3 of
\cite{Ramanan} ). By Lemma 3.1 of \cite{Sun05}, $\mathcal{E}$ is a
family of stable bundles of rank $2$ and with fixed determinant
$\text{det}(L_{1})\otimes \text{det}(L_{2})\cong\mathcal{L}$, which
defines a morphism
\begin{equation}
\Psi_{(L_{1}, L_{2})}:\ P_{(L_{1},L_{2})}\longrightarrow SU_{C}(2,\mathcal{L})=M.
\label{morphism:psi}
\end{equation}
Let $l$ be a line in $P_{(L_{1},L_{2})}$ passing through
 \[\begin{CD}
q=[0@>>>L_{1}@>\iota>>W@>>>L_{2}@>>>0],
\end{CD} \]
then
\begin{equation}
\Psi_{(L_{1}, L_{2})}|_{l}: l\longrightarrow M=SU_{C}(2,\mathcal{L})
\end{equation}
is a rational curve of split type passing through the point $[W]\in
M$, which has degree $2$.
\end{proof}

When $g=2$, the same as Lemma \ref{lm:s(w)=2}, we have:
\begin{Lemma}
If $g=2$, $r=2$ and $d$ is odd, for any $[W]\in M$, $s(W)=1$.
\label{lm:s(w)=1}
\end{Lemma}

By Lemma \ref{lm:s(w)=1} and Proposition \ref{prop:M1}, we have:
\begin{Proposition}
If $g=2$, $r=2$ and $d$ is odd, then, for any $[W]\in M$, there
exist a rational curve of split type passing through it, which has
degree $2$.
\end{Proposition}

\section*{Acknowledgments}
The author is grateful to her supervisor Prof. Xiaotao Sun for his helpful suggestions in the preparation of this paper
and to Prof. Meng Chen and Prof. Kejian Xu for their help.

\end{document}